\documentclass{amsart}
\usepackage[T1]{fontenc}
\usepackage{amsmath}
\usepackage{amssymb}
\usepackage{amsfonts}
\usepackage{graphicx}

\usepackage{lmodern}
\usepackage{microtype}
\usepackage{mathtools}
\usepackage[shortlabels]{enumitem}
\usepackage[hidelinks]{hyperref}

\allowdisplaybreaks
\numberwithin{equation}{section}

\newcommand{\M}{\mathcal M}
\newcommand{\one}{\mathbf 1}
\newcommand{\tr}{\tau}

\newcommand{\ind}{\mathbf 1}

\newtheorem{theorem}{Theorem}[section]
\newtheorem{lemma}[theorem]{Lemma}
\newtheorem{proposition}[theorem]{Proposition}
\newtheorem{cor}[theorem]{Corollary}

\theoremstyle{definition}

\newcommand\F{\mathcal{F}}
\newcommand\E{\mathbb{E}}
\newcommand\R{\mathbb{R}}

\numberwithin{equation}{section}

\begin{document}

\title[Burkholder inequalities]{Best constants in noncommutative Burkholder inequalities}

\author{Yong Jiao}
\address{School of Mathematics and Statistics, Central South University, Changsha 410085,
People's Republic of China}
\email{jiaoyong@csu.edu.cn}

\author{Adam Os\k ekowski}
\address{Faculty of Mathematics, Informatics and Mechanics\\
 University of Warsaw\\
Banacha 2, 02-097 Warsaw\\
Poland}
\email{ados@mimuw.edu.pl}

\author{Lian Wu}
\address{School of Mathematics and Statistics, Central South University, Changsha 410085,
People's Republic of China}
\email{wulian@csu.edu.cn}

\author[Y. Zuo]{Yahui Zuo}
\address{School of Mathematics and Computational Science, Xiangtan University, Xiangtan 411105,
People's Republic of China}
\email{yahuizuo@xtu.edu.cn}

\thanks{This work was partially supported by the NSFC (nos.12125109, W2411005, 12361131578, 12401166), the Natural Science Foundation of Hunan Province
(no. 2025ZYJ002), Narodowe Centrum Nauki (Poland), grant 2023/48/Q/ST1/00013 and  the Science and Technology Innovation Program of Hunan Province
(no. 2024RC1016).}


\begin{abstract}
The paper contains the study of noncommutative Burkholder inequalities. We prove that the constant involved is of order $O(p/\log p)$ when $p\to \infty$, which is known to be optimal even in the commutative setting.
This answers an open problem raised by Randrianantoanina  in 2007, and by Junge and Xu in 2008.
\end{abstract}

\maketitle

\section{Introduction}
The purpose of this paper is to study the sharp version of Burkholder inequality for  noncommutative martingales.
To put our results in an appropriate context, let us present the milestones of the development of the research in the classical case. Inspired by searching for Banach spaces linearly isomorphic to a complemented subspace of an $L^p$-space, Rosenthal \cite{Ro} established the following statement: if $2\leq p<\infty$, then for any sequence of independent mean-zero random variables $d_1$, $d_2$, $\cdots$ belonging to $L^p$ we have the two-sided estimate
\begin{equation}\label{rosenthal_classic}
\begin{split}
&\alpha_p^{-1}\left\{\bigg(\sum_{k=1}^\infty \E d_k^2\bigg)^{1/2}+\left(\E \sum_{k=1}^\infty |d_k|^p\right)^{1/p}\right\}\\
&\leq \left(\E \bigg|\sum_{k=1}^\infty d_k\bigg|^p\right)^{1/p}\leq \beta_p\left\{\bigg(\sum_{k=1}^\infty \E d_k^2\bigg)^{1/2}+\left(\E \sum_{k=1}^\infty |d_k|^p\right)^{1/p}\right\},
\end{split}
\end{equation}
where $\alpha_p$ and $\beta_p$ are constants depending only on $p$. This inequality, named after Rosenthal, nowadays has become a very useful probalistic tool and has found many generalizations and applications (e.g. \cite{AS2005,AS3,CD1,FM,H,HMs,JSXZ,JS,M}).

It is well known that the order of the best constants in probabilistic inequalities usually provide important additional insights. Based on this, the constants involved in \eqref{rosenthal_classic} have  been extensively studied.
The proof of Rosenthal \cite{Ro} gives \eqref{rosenthal_classic} with $\alpha_p=4$ and $\beta_p$ growing exponentially as $p\to \infty$. Johnson, Schechtman and Zinn \cite{JSZ} improved the result and proved that the optimal order of $\beta_p$ as $p\to \infty$ is $O(p/\log p)$. This statement was further generalized by Talagrand \cite{Ta} to the case in which the independent random variables $(d_n)_{n\geq 1}$ take values in a given Banach space. This vector-valued version was re-proved later by Kwapie\'n and Szulga \cite{KS} by means of hypercontractivity-type arguments. Replacing $L^p$-space by a symmetric Banach space $X$ with the Kruglov property, Astashkin and Sukochev \cite{AS-aop} discovered that the sharp constant involved in Rosenthal inequality is equivalent to the norm of the Kruglov operator in $X.$

In 1973, Burkholder proposed an extension of Rosenthal inequality to the martingale setting. It was shown in \cite{Bu0} that if $2\leq p<\infty$ and $d_1$, $d_2$, $\cdots$ is an $L^p$ bounded martingale difference sequence (adapted to some discrete-time filtration $(\F_n)_{n\geq 1}$), then
\begin{equation}\label{burkholder_classic}
\begin{split}
&A_p^{-1}\left\{\left(\E \bigg(\sum_{k=1}^\infty\E_{k-1}d_k^2\bigg)^{p/2}\right)^{1/p}+\left(\E \sum_{k=1}^\infty |d_k|^p\right)^{1/p}\right\}\\
&\leq \left(\E \bigg|\sum_{k=1}^\infty d_k\bigg|^p\right)^{1/p}\!\!\leq B_p\left\{\left(\E \bigg(\sum_{k=1}^\infty\E_{k-1}d_k^2\bigg)^{p/2}\right)^{1/p}\!\!+\left(\E \sum_{k=1}^\infty |d_k|^p\right)^{1/p}\right\}.
\end{split}
\end{equation}
Here $\E_n$ denotes the conditional expectation with respect to the $\sigma$-algebra $\F_n$ and  $A_p$ and $B_p$ are constants depending only on $p$. The above estimate is usually referred to as the Burkholder inequalities. Burkholder's approach, based on  extrapolation technique (good-$\lambda$ inequalities), gave $A_p$ of order $\sqrt{p}$ and $B_p$ of exponential growth as $p\to \infty$. While the order of $A_p$ was the best possible, the information on $B_p$ was far from being optimal. A careful inspection of some results appearing in Garsia's book \cite{Ga} reveals that the right inequality in \eqref{burkholder_classic} holds with $B_p\leq 10p$. This linear growth was finally improved by Hitczenko \cite{Hi} to the optimal $O(p/\log p)$, the same as in the case of independent random variables \cite{JSZ}. See also Os\k ekowski \cite{O} for a different proof of this fact, exploiting the so-called Burkholder's (or Bellman function) method.

In recent years, a trend in the general study of probabilistic inequalities is to find appropriate analogues of classical inequalities in the context of noncommutative $L^p$ spaces. The paper which started the rapid development of this area is the seminal work of Pisier and Xu \cite{PX}, containing the noncommutative counterparts of Burkholder-Gundy inequalities. Since then, many important estimates have been successfully transferred to the noncommutative realm. For example, the noncommutative  Doob's maximal $L^p$ estimate was obtained by Junge \cite{J} and some crucial noncommutative weak-type inequalities were proposed in \cite{R1,R2,R21,R3}. We refer to \cite{BCO,BCPY,CP,HJP1, HJP2,HM,JSZZ,JZWZ,PR,Pe,RWX} and references therein for recent progress on this topic. In particular, noncommutative Burkholder/Rosenthal inequalities, which play a central role in this present paper, were first investigated by Junge and Xu in \cite{JX,JX3}. We refer to \cite{DPS,JiaoSZ,JPX,RW1,RW2,SZ} for the generalizations of noncommutative Burkholder/Rosenthal inequality.

With the development of noncommutative probabilistic inequalities, a natural question is to determine the optimal orders of the constants in these inequalities. Junge and Xu \cite{JX2} discussed the optimal orders of the best constants in the noncommutative Burkholder-Gundy, Doob and Stein inequalities. The only case left unsolved in \cite{JX2} was settled by Randrianantoanina in \cite{R21}. As for the noncommutative Burkholder/Rosenthal inequality, Junge and Xu \cite{JX} established an appropriate analogue of \eqref{burkholder_classic} with both $A_p$, $B_p$ of order $O(p^2)$ as $p\to \infty$.
The above orders of constants were later improved to linear by Randrianantoanina \cite{R3}: while this order is optimal for $A_p$, it is still an open problem whether the same is true for $B_p$. Indeed, this problem is asked by Randrianantoanina  in \cite[Remark 4.6]{R3} (see also Junge and Xu \cite[Remark 2.2]{JX3}).

It is worth mentioning that some positive results in this direction have been obtained a few years ago by Junge and Zeng in \cite{JZ}. Suppose that the operators $d_1$, $d_2$, $\cdots$ are trace-zero and full independent over $\tau$. Then their result can be reformulated as follows:
$$ \left(\tau \bigg|\sum_{k=1}^\infty d_k\bigg|^p\right)^{1/p}\!\!\leq C\left\{\sqrt{p}\bigg(\sum_{k=1}^\infty \tau \left(|d_k|^2\right)\bigg)^{1/2}\!\!+p\left(\tau \left(\sum_{k=1}^\infty |d_k|^p\right)\right)^{1/p}\right\},$$
where $C$ is an absolute constant. The orders $O(\sqrt{p})$ and $O(p)$ are optimal even in the classical case (see \cite{Pi}). The proof of this estimate rested on an appropriate version of the ``$p\to 2p$'' argument, which turn out to work very efficiently for the above estimate (i.e., it returns the constants of optimal order).

The purpose of this paper is to settle the problem of  Randrianantoanina \cite[Remark 4.6]{R3} by establishing the following sharp noncommutative Burkholder inequality. Let $(\mathcal{M},\tau)$ be a von Neumann algebra, paired with a trace satisfying some continuity conditions. Let $(\mathcal{M}_k)_{k\geq 1}$ be an increasing sequence of von Neumann subalgebras of $\mathcal{M}$ such that the union $\bigcup_{k\geq 1} \mathcal{M}_k$ is weak$^{\ast}$ dense in $\mathcal{M}.$ For $k\geq 	1$, denote by $\mathcal{E}_{k}$ the conditional expectation with respect to $\mathcal{M}_k.$

\begin{theorem}\label{main:result}
Let $2\le p<\infty$ and let $d_1$, $d_2$, $\cdots$ be an $L^p$ bounded martingale difference sequence adapted to $(\mathcal M_n)_{n\geq 1}$. Then, with the convention $\mathcal E_0=\mathcal E_1$, we have
\begin{equation*}\label{eq:main}
\begin{split}
\left(\tau \bigg|\sum_{k=1}^\infty d_k\bigg|^p\right)^{1/p} \leq C_{\rm abs}\frac{p}{\log p}&\left\{\left(\tau \bigg(\sum_{k=1}^\infty\mathcal E_{k-1}|d_k|^2\bigg)^{p/2}\right)^{1/p}\right.\\
&\left.+\left(\tau \bigg(\sum_{k=1}^\infty\mathcal E_{k-1}|d_k^*|^2\bigg)^{p/2}\right)^{1/p}+\left(\tau\sum_{k=1}^\infty |d_k|^p\right)^{1/p}\right\}.
\end{split}
\end{equation*}
Moreover, the order
$p/\log p$ is optimal as $p\to \infty.$
\end{theorem}

The proof combines a noncommutative Bennett-type estimate with a dyadic localization argument. The triple Golden–Thompson inequality yields an exponential trace bound, from which we derive a moment estimate with the optimal $p/\log p$ growth while retaining an $L^2$-energy factor. To apply this estimate, we truncate and center the martingale differences at dyadic levels and use Cuculescu projections associated with their conditioned square functions to construct stopped approximations. Differences between successive approximations have uniformly controlled increments and conditioned square functions, and their $L^2$-energies are estimated by spectral tails of the original conditioned square function and dyadic portions of the martingale differences. These estimates can then be summed over all scales. A key step is to recombine the localized pieces in $L^{2p}$, where a summable convolution kernel prevents any additional loss in the dependence on $p$. Finally, a distribution-function comparison controls the stopping and truncation errors and yields the desired Burkholder inequality for which the optimality follows from the classical commutative case.

The paper is organized as follows. In section 2, we  introduce the notation and recall the necessary facts about noncommutative \(L^p\)-spaces, distribution functions, and martingales, together with the triple Golden–Thompson inequality. In section 3, we establish a Bennett-type  estimate. Section 4 contains the construction of the dyadic stopped approximations and the proof of the localization estimates for their successive differences. Section 5 is devoted to establishing the dyadic summability bounds and the \(L^{2p}\) recombination argument. Finally, the last section compares the original martingale with its stopped approximations and combines the preceding estimates to prove Theorem 1.1.

Throughout, we always use the letter $C_{\rm abs}$ to denote an absolute constant which may vary from line to line.

\section{Preliminaries}
Throughout the paper, we use standard notation from the theory of operator algebras; for the detailed exposition of the subject we refer to the monographs \cite{KR1,KR2,T}. For a given Hilbert space $H$, the symbol $B(H)$ will stand for the algebra of all bounded operators acting on $H$.
Let $\mathcal{M}$ be a von Neumann subalgebra of $B(H)$, equipped with a finite normal faithful trace $\tau$. We will assume that $\tau$ is normalized (i.e., we have $\tau({\bf 1})=1$, where ${\bf 1}$ is the identity operator) and sometimes refer to $(\mathcal{M},\tau)$ as  a noncommutative probability space. The set of all $\tau$-measurable operators will be denoted by $L^0(\mathcal{M})$. For any $\tau$-measurable self-adjoint operator $x$ and any Borel subset $B$ of $\R$, the spectral
projection of $x$ corresponding to the set $B$ will be denoted by $\chi_B(x)$. For a
given operator $x \in \mathcal M$, we denote by $l(x)$ (resp. $r(x)$) the left (resp. right) support
projection of $x$ (see, e.g., \cite{T} page 134, for definitions). If $x=x^*\in \mathcal M$, then $l(x)$ and $r(x)$ coincide, and we denote it by ${\rm supp}(x)$.

The following basic lemma will be used later. The proof is straightforward and we leave details to the reader.

\begin{lemma}\label{basic:proj}
If $p$ is a projection in $\mathcal M$, $q=\one-p$, and $z=z^*\in L^2(\mathcal M)$, then
$$\|z-pzp\|_2^2=2\|pzq\|_2^2+\|qzq\|_2^2\leq 2\tau(qz^2).$$
\end{lemma}

Let $x \in L^0(\mathcal{M})$.  The distribution function of $x$ is defined by
$$\lambda_s(x)=\tau\left(\chi_{(s,\infty)}(|x|)\right), \quad 0\leq s<\infty.$$
For $x, y \in L^0(\mathcal M)$, we have
\begin{equation}\label{distribution-triangle}
\lambda_{s+t}(x+y)\leq \lambda_s(x)+\lambda_t(y),\quad 0\leq s,t<\infty.
\end{equation}
If $x\leq y$ are bounded self-adjoint operators, then
\begin{equation}\label{eq:ordered-tail}
\tau(\chi_{(s,\infty)}(x))
\leq \tau(\chi_{(s,\infty)}(y)),\quad -\infty<s<\infty.
\end{equation}
For $0< p\leq \infty$, we define the associated noncommutative $L^p$-space $L^p(\mathcal M)$ as the set of all measurable operators $x$ such that
$$ \|x\|_p:=(\tau(|x|^p))^{1/p}<\infty,$$
where $|x|=(x^*x)^{1/2}$ is the modulus of $x$.
Recall that for $p=\infty$, the space $L^p(\mathcal{M})$ coincides with $\mathcal{M}$ with its usual operator norm. Given $0< p< \infty$ and $x\in L^p(\mathcal M)$,  the $L^p$-norm of $x$ can be expressed in terms of its distribution function:
$$\|x\|_p^p=p\int_0^\infty s^{p-1}\lambda_s(x)\ ds.$$

\begin{lemma}\label{lem:tails}
For $x\in \mathcal M$, $y\in L^0(\mathcal M)$, we have
\begin{equation}\label{eq:perturb-tail}
\lambda_s(x+y)\leq \tau(r(x))+\lambda_s(y),\quad s>0.
\end{equation}
If in addition $x=x^*$ and $qxq=0$ for some projection $q\in\mathcal M$, then
\begin{equation}\label{eq:rank-bound}
\tau({\rm supp}(x))\leq 2\tau({\bf 1}-q).
\end{equation}
\end{lemma}

\begin{proof}
Given $s>0$, set $e=\chi_{[0,s]}(|y|)\wedge({\bf 1}-r(x))$. Then  $xe=0$ and
$$\|(x+y)e\|_\infty\leq s.$$
Using \eqref{distribution-triangle} and the above inequality, we get
\begin{align*}
\lambda_{s}(x+y)&\leq \lambda_s((x+y)e)+\lambda_0((x+y)(1-e))\\
&=\lambda_0((x+y)({\bf 1}-e))\\
&\leq \tau({\bf 1}-e) ) \leq \lambda_s(y)+\tau(r(x)),
\end{align*}
This verifies the estimate \eqref{eq:perturb-tail}.

We now show \eqref{eq:rank-bound}. Put $f={\bf 1}-q$. Since $qxq=0$, we may write
\[
x=fx+qxf.
\]
Obviously, $l(fx)\leq f$  and   $r(qxf)\leq f$. Therefore,
$$\tau(l(x))\leq \tau(l(fx))+\tau(l(qxf))=\tau(l(fx))+\tau(r(qxf))\leq  2\tau(f).$$
For self-adjoint $x$,
$l(x)=r(x)={\rm supp}(x)$, proving \eqref{eq:rank-bound}.
\end{proof}

We record below the triple Golden-Thompson inequality; see \cite[Theorem~A]{Rand24}  for the proof. Indeed, in \cite{Rand24}, this theorem is stated for a tracial state and for wider
measurable operators.

\begin{lemma}\label{triple-GT}
For self-adjoint operators $a,b,c\in\mathcal M$,
\begin{equation}\label{eq:triple-GT}
\tau(e^{a+b+c})\leq
\int_0^\infty\tau\left(
 e^{c/2}(e^{-a}+t{\bf 1})^{-1}e^b
 (e^{-a}+t{\bf 1})^{-1}e^{c/2}\right)\,dt.
\end{equation}
\end{lemma}

We now turn our attention to the general setup of noncommutative martingales. Suppose that $(\mathcal{M}_n)_{n\geq 1}$ is a filtration, i.e., a nondecreasing sequence of von Neumann subalgebras of $\mathcal{M}$ whose union is weak$^*$-dense in $\mathcal{M}$. Then for any $n\geq 1$ there is a normal conditional expectation $\mathcal{E}_n$ from $\mathcal{M}$ onto $\mathcal{M}_n$, satisfying
\begin{enumerate}[{\rm (i)}]
\item $\mathcal{E}_n(axb)=a\mathcal{E}_n(x)b$ for all $a,\,b\in\mathcal{M}_n$ and $x\in \mathcal{M}$;
\item $\tau\circ \mathcal{E}_n=\tau$.
\end{enumerate}
It is straightforward to check that the conditional expectations satisfy the tower property $\mathcal{E}_m\mathcal{E}_n=\mathcal{E}_n\mathcal{E}_m=\mathcal{E}_{\min(m,n)}$ for all  $m$ and $n$. Furthermore, since $\mathcal{E}_n$ is trace preserving, it can be extended to a contractive projection from $L^p(\mathcal{M},\tau)$ onto $L^p(\mathcal{M}_n,\tau_n)$ for all $1\leq p\leq \infty$, where $\tau_n$ is the restriction of $\tau$ to $\mathcal{M}_n$.

A sequence $(x_n)_{n\geq 1}$ is called a noncommutative  martingale with respect to $(\mathcal M_n)_{\geq 1}$ if $x_n\in L^1(\mathcal{M}_n)$ for all $n\geq 1$ and
$$\mathcal{E}_{n-1}(x_n)=x_{n-1},\quad n\geq 1.$$
Let $1\leq p\leq \infty$ and let $(x_n)_{n\geq 1}\subset L_p(\mathcal M)$ be a noncommutative martingale. We say that $(x_n)_{n\geq 1}$ is  $L_p$-bounded if
$$\sup_{n\geq 1}\|x_n\|_p<\infty.$$
The difference sequence $dx=(dx_n)_{n\geq 1}$ of a martingale $x=(x_n)_{n\geq 1}$ is defined by
$$dx_1=x_1,\quad dx_n=x_{n}-x_{n-1}.$$
Given a martingale $x = (x_n)_{n \geq 1}$ in $L^2(\mathcal M)$, we set
 \[s_{c,n} (x) = \Big ( \sum_{k=1}^n \mathcal E_{k-1}|dx_k |^2 \Big )^{1/2},\quad s_c (x) = \Big ( \sum_{k \geq 1} \mathcal E_{k-1}|dx_k |^2 \Big )^{1/2}\]
and
 \[s_{r,n} (x) = \Big ( \sum_{k=1}^n \mathcal E_{k-1}| dx_k^* |^2 \Big)^{1/2},\quad s_r (x) = \Big ( \sum_{k \geq1} \mathcal E_{k-1}| dx_k^* |^2 \Big)^{1/2},
 \]
with the convention that $\mathcal E_0=\mathcal E_1$. These are called the column and row conditioned square functions, respectively.  For self-adjoint martingales $x=(x_n)_{n\geq 1}$, we  simplicity write $s_n(x)$ and $s(x)$, instead.

\section{A Bennett type estimate}
The purpose of this section is to obtain a moment estimate for self-adjoint martingales whose differences and conditioned square functions are bounded in operator norm. Using the triple Golden–Thompson inequality, we first establish an exponential trace bound with a compensating conditional variance term. We then derive a Bennett-type tail estimate that retains the \(L^2\)-energy of the terminal value. Integration yields a moment bound with \(p/\log p\) growth, whose \(L^2\)-energy factor will be essential when summing the localized estimates over dyadic scales.

\begin{lemma}\label{lem:exp}
Let  $M>0$, $\theta>0$ and $\kappa= \frac{e^{\theta M}-1-\theta M}{M^{2}}$. Suppose that $y=(y_k)_{1\leq k\leq N}$ is a finite self-adjoint martingale with $y_1=0$ and
$\sup_{1\leq k\leq N}\|dy_k\|_\infty \leq M$.
Then, for every $1 \leq n\leq N$,
\begin{equation}\label{eq:comp-exp}
\tau\big(\exp\left(\theta y_n -\kappa s_n^2(y)\right)\big)\leq\tau({\bf 1}).
\end{equation}
The same assertion holds with $y$ replaced by $-y$.
\end{lemma}
\begin{proof}
For any $s\in\mathbb R$ with $|s|\leq M$, the integral Taylor formula gives
\begin{align*}
e^{\theta s}-1-\theta s
&=\theta^2s^2\int_0^1(1-u)e^{u\theta s}\,du\\
&\leq\theta^2s^2\int_0^1(1-u)e^{u\theta M}\,du
=\kappa s^2.
\end{align*}
By functional calculus, for  every $2\leq k\leq N$,
\begin{equation}\label{eq:one-step}
\mathcal E_{k-1}(e^{\theta dy_k})\leq{\bf 1}+\kappa\cdot \mathcal E_{k-1}(dy_k^2)\leq \exp\{\kappa\cdot \mathcal E_{k-1}(dy_k^2)\}.
\end{equation}
For $t\in (0,\infty)$, set $R_t:=(\exp\{\kappa\mathcal E_{k-1}(dy_k^2)\}+t{\bf 1})^{-1}$. Applying Lemma \ref{triple-GT} with
$$a=-\kappa\cdot \mathcal E_{k-1}(dy_k^2),\quad b=\theta dy_k,\quad c=H:=\theta y_{k-1}-\kappa  s_{k-1}^2(y),$$
and using \eqref{eq:one-step}, we obtain
\begin{align*}
\tau\big(\exp\left(\theta y_k -\kappa s_k^2(y)\right)\big)
&\leq \int_0^\infty
 \tau\left(e^{H/2} R_t e^{\theta dy_k}R_t e^{H/2}\right)\,dt\\
&=\int_0^\infty
 \tau\left(e^{H/2}R_t\mathcal E_{k-1}(e^{\theta dy_k})R_t e^{H/2}\right)\,dt\\
 &\leq\int_0^\infty
 \tau\left(e^{H/2}R_t \exp\{\kappa\cdot \mathcal E_{k-1}(dy_k^2)\} R_t e^{H/2}\right)\,dt.
\end{align*}
Since $\int_0^\infty R_t \exp\{\kappa\cdot \mathcal E_{k-1}(dy_k^2)\} R_t\,dt={\bf 1}$ (in the norm convergence), the right side equals $\tau(e^H)$. Namely, we have
\begin{align*}
\tau\big(\exp\left(\theta y_k-\kappa s_k^2(y)\right)\big)
&\leq \tau\big(\exp\left(\theta y_{k-1}-\kappa s_{k-1}^2(y)\right)\big).
\end{align*}
Note that $y_1=s_1(y)=0$. By iteration, we conclude the desired result \eqref{eq:comp-exp}. Applying the
same argument to $-y$ proves the last assertion.
\end{proof}

Using Lemma \ref{lem:exp}, we may establish the following Bennett type estimate.

\begin{proposition}\label{lem:intrinsic}
Let $a>0$ and let $y=(y_n)_{1\leq n\leq N}$ be a self-adjoint martingale
with $y_1=0$ such that
\begin{equation}\label{eq:intrinsic-assumptions}
 \sup_{1\leq n\leq N}\|dy_n\|_\infty\leq3a,
 \qquad
 s_N^2(y)\leq5a^2{\bf 1}.
\end{equation}
Then for any $4\leq p<\infty$,
\begin{equation}\label{eq:intrinsic-moment}
\|y_N\|_p^p\leq \left(\frac{32p}{\log p}\right)^p a^{p-2}\|y_N\|_2^2.
\end{equation}
\end{proposition}

\begin{proof} There is nothing to prove if $\|y_N\|_2^2=\tau(s_N^2(y))=0$. Therefore we assume that  $\|y_N\|_2^2>0$. For
$s\in\mathbb R$, define $\varphi(s)=e^s-s-1$. Then $\varphi$ is nonnegative on $\mathbb R$ and increasing on $[0,\infty)$. Put
$$\theta:=\frac{\log p}{3a},\qquad \kappa :=\frac{p-1-\log p}{9a^2},\qquad H_{\pm}:=\pm \theta y_N-\kappa s_N^2(y).$$ Then by Lemma \ref{lem:exp} and $\tau(y_N)=\tau(y_1)=0$, we have
\begin{equation}\label{varphi-x-N}\tau(\varphi(H_{\pm}))=\tau(e^{H_{\pm}})-\tau({\bf 1})-\tau(H_{\pm})\leq -\tau(H_{\pm}) = \kappa \|y_N\|_2^2.
\end{equation}
Let $t>0$ be such that $s:=\theta t-5\kappa a^2>0.$ By \eqref{eq:intrinsic-assumptions},
$$H_{\pm}\geq \pm \theta y_N-5\kappa a^2 {\bf 1}.$$
Hence,
\begin{align*}
\tau(\chi_{(t,\infty)}(\pm y_N))&=\tau(\chi_{(s,\infty)}(\pm\theta y_N-5\kappa a^2 {\bf 1}))\leq \tau(\chi_{(s,\infty)}(H_{\pm}))
\end{align*}
Applying Chebyshev's inequality and \eqref{varphi-x-N}, we arrive at
\begin{equation*}\label{equ:aftercheby}
\tau(\chi_{(t,\infty)}(\pm y_N))\leq \frac{\tau\varphi(H_{\pm})}{\varphi(s)}\leq  \frac{\kappa \|y_N\|_2^2}{\varphi(\theta t-5\kappa a^2)}.
\end{equation*}
Consequently,
$$\lambda_{t}(y_N)\leq \tau(\chi_{(t,\infty)}(y_N))+\tau(\chi_{(t,\infty)}(- y_N))\leq \frac{2\kappa \|y_N\|_2^2}{\varphi(\theta t-5\kappa a^2)}.$$
Set $t_0:=(5\kappa a^2+2)/\theta$. Note that  $\varphi(s)\geq e^s/2$ for  $s\geq 2$. Thus,
$$\lambda_{t}(y_N)\leq  4\kappa e^{5\kappa a^2-\theta t} \|y_N\|_2^2,\qquad t\geq t_0.$$
We now can estimate $\|y_N\|_p$ as follows:
\begin{equation}\label{equ:norm-of-x}
\begin{split}
\|y_N\|_p^p&=p \int_0^\infty t^{p-1}\lambda_{t}(y_N)\ dt\\
& \leq p\|y_N\|_2^2\int_0^{t_0} t^{p-3}\ dt + 4\kappa e^{5\kappa a^2}  p \|y_N\|_2^2  \int_{t_0}^\infty  e^{-\theta t} \ dt\\
&\leq \|y_N\|_2^2 \left\{\frac{p}{p-2}t_0^{p-2}+4\kappa e^{5\kappa a^2}\Gamma(p+1)\theta^{-p} \right\},
\end{split}
\end{equation}
where we use the Chebyshev inequality $\lambda_t(y_N)\leq t^{-2}\|y_N\|_2^2$ in the second inequality.
Note that
$$\kappa\leq \frac{p}{9a^2},\quad 5\kappa a^2\leq \frac{5p}{9},\quad t_0\leq\frac{a}{\log p}\left(\frac{5p}{3}+6\right)\leq 4a \frac{p}{\log p}.$$
Moreover, $\Gamma(p+1)\leq (2p)^p$ for $p>4$. Substituting these facts into \eqref{equ:norm-of-x}, we conclude the desired estimate
\eqref{eq:intrinsic-moment}.
\end{proof}

\section{A localization estimate}
We  construct in this section localized martingales to which the estimates of Section 3 apply. At each dyadic level, we truncate and center the original martingale differences and then use Cuculescu projections associated with the resulting conditioned square function to form a stopped approximation. The differences between consecutive approximations have bounded increments and bounded conditioned square functions at the corresponding scale. The main task is to control their \(L^2\)-energies by spectral tails of the original conditioned square function and the contributions of the original differences on dyadic spectral intervals.

Let $4\leq p<\infty$ and let $x=(x_n)_{1\leq n\leq N}$ be a finite self-adjoint $L^p$-martingale with $x_1=0$.  
 For any $1\leq n\leq N$ and any $j\in\mathbb Z$, we define
\begin{equation*}\label{eq:truncation}
 h_n^{(j)}=dx_n\ind_{[0,2^j]}(|dx_n|),
 \qquad
v_n^{(j)}=\sum_{k=1}^n h_k^{(j)}-\mathcal E_{k-1}(h_k^{(j)}).
\end{equation*}
Obviously, for any fixd $j\in\mathbb Z$, $v^{(j)}=(v_n^{(j)})_{1\leq n\leq N}$ forms a martingale with the difference sequence given by
$$dv_1^{(j)}=0, \quad dv_n^{(j)}=h_n^{(j)}-\mathcal E_{n-1}(h_n^{(j)})\,\,  (2\leq n\leq N).$$  The following properties of $v^{(j)}$ can be easily verified.

\begin{lemma}\label{basic-lem-1}
Fix $j\in\mathbb Z$. For any $1\leq n\leq N$,
\begin{enumerate}[\rm (i)]
\item $\|dv_n^{(j)}\|_\infty\leq2^{j+1};$
\item $\mathcal E_{n-1} [(dv_n^{(j)})^2]=\mathcal E_{n-1} [(h_n^{(j)})^2]-[\mathcal E_{n-1} (h_n^{(j)})]^2\leq 2^{2j}\one;$
\item $\mathcal E_{n-1} [(dv_n^{(j)})^2]\leq  \mathcal E_{n-1} [(dx_n)^2]$ and $s_N^2(v^{(j)})\leq  s_N^2(x).$
\end{enumerate}
\end{lemma}

For each fixed $j\in \mathbb Z$,  we introduce the Cuculescu projections generated from the submartingale $(s_n^2(v^{(j)}))_{1\leq n \leq N}$: set $q_1^{(j)}=\one$ and inductively define
\begin{equation*}\label{eq:q-definition}
 q_n^{(j)}
 =\ind_{[0,2^{2j}]}
 \bigl(q_{n-1}^{(j)}s_n^2(v^{(j)}) q_{n-1}^{(j)}\bigr)q_{n-1}^{(j)},\qquad 2\leq n\leq N.
\end{equation*}
For any $2\leq n\leq N$, we set
$$p_n^{(j)}=q_{n-1}^{(j)}-q_n^{(j)}.$$
Then the projections $p_2^{(j)},\ldots,p_N^{(j)},q_N^{(j)}$ are mutually
orthogonal and sum to $\one$. We collect below some fundamental
properties of $(q_n^{(j)})$ and $(p_n^{(j)})$.

\begin{lemma}\label{lem-cuculescu} Fix $j\in\mathbb Z$. For any $2\leq n\leq N$,
\begin{enumerate}[\rm (i)]
\item $q_n^{(j)}\in\M_{n-1}$;
\item $q_n^{(j)}\leq q_{n-1}^{(j)}$;
\item $q_n^{(j)}s_n^2(v^{(j)}) q_n^{(j)}\leq 2^{2j} q_n^{(j)};$
\item $p_n^{(j)}s_n^2(v^{(j)})p_n^{(j)}\geq 2^{2j}p_n^{(j)};$
\item $\tau(\one-q_N^{(j)})\leq 2^{1-2j}\tau\big(s_N^2(v^{(j)})\one_{(2^{2j-1},\infty)}(s_N^2(v^{(j)}))\big).$
\end{enumerate}
\end{lemma}

\begin{proof} Properties (i)-(iv) are now well-known; see e.g. \cite{R1,R21} for the proof. Therefore, we only check (v) here. For simplicity, we suppress the superscript $(j)$.
 By (iv), $p_ns_n^2(v)p_n\geq 2^{2j}p_n.$ Therefore,
\begin{align*}
\tau(\one-q_N)=\sum_{n=2}^N \tau(p_n)&\leq 2^{-2j} \sum_{n=2}^N\tau(p_ns_n^2(v)p_n)\\
&\leq 2^{-2j} \sum_{n=2}^N\tau(p_ns_N^2(v)p_n)=2^{-2j}\tau((\one-q_N)s_N^2(v))
\end{align*}
Splitting $s_N^2(v)$ at $2^{2j-1}$, we get
\begin{align*}
 \tr((\one-q_N)s_N^2(v))&=\tr((1-q_N)s_N^2(v)\one_{[0,2^{2j-1})}(s_N^2(v)))\\
 &\qquad+\tr((\one-q_N)s_N^2(v)\one_{[2^{2j-1},\infty)}(s_N^2(v)))\\
 &\leq2^{2j-1}\tr(\one-q_N)+\tr(s_N^2(v)\one_{[2^{2j-1},\infty)}(s_N^2(v))).
\end{align*}
Combining the above two estimates, we conclude
$$\tau(\one-q_N)\leq \frac12 \tau(\one-q_N)+2^{-2j}\tr(s_N^2(v)\one_{[2^{2j-1},\infty)}(s_N^2(v))),$$
which implies the desired assertion.
\end{proof}

%
Obviously, the sequence
$(q_{n-1}^{(j)}dv_n^{(j)}q_{n-1}^{(j)})_{1\leq n\leq N}$ forms a martingale difference sequence. The associated martingale $w^{(j)}=(w_n^{(j)})_{1\leq n\leq N}$ is given by
\begin{equation*}\label{eq:stopped-martingale}
 w_1^{(j)}=0,\qquad w_n^{(j)}=\sum_{k=2}^n q_{k-1}^{(j)}dv_k^{(j)}q_{k-1}^{(j)}.
\end{equation*}
Then we find the following useful properties of $w^{(j)}$.

\begin{lemma}\label{lem:stopping}
For every $j\in\mathbb Z$, we have
\begin{enumerate}[\rm (i)]
\item $s_N^2(w^{(j)})\leq 2^{2j+1}\one;$
\item $q_N^{(j)}w_N^{(j)}q_N^{(j)}=q_N^{(j)}v_N^{(j)}q_N^{(j)}.$
\end{enumerate}
\end{lemma}

\begin{proof}
Fix $j\in\mathbb Z$. For simplicity, we suppress the superscript $(j)$.
Since  $q_N\leq q_{n-1}$ for every $2 \leq n\leq N$, the second item follows immediately. Let us now verify (i).   Since $q_n$ is a spectral projection of
$q_{n-1}s_n^2(v)q_{n-1}$ in the algebra $q_{n-1}\mathcal M_{n-1} q_{n-1}$, we know that $q_n$ commutes with $q_{n-1}s_n^2(v)q_{n-1}$. Moreover, by Lemma \ref{lem-cuculescu}(ii), $q_n\leq q_{n-1}$. Therefore,
\[
 q_{n-1}s_n^2(v)q_{n-1}=p_n s_n^2(v) p_n+q_n s_n^2(v) q_n.
\]
Subtracting $q_{n-1}s_{n-1}^2(v) q_{n-1}$ and summing gives the exact identity
\begin{equation}\label{eq:stopping-identity}
 \sum_{n=2}^m q_{n-1}\mathcal E_{n-1} [(dv_n)^2]q_{n-1}
 =\sum_{n=2}^m p_n s_n^2(v)p_n+q_ms_m^2(v)q_m,
 \qquad 2\leq m\leq N.
\end{equation}
By Lemma \ref{lem-cuculescu}(iii), the second term on the right is bounded by $2^{2j}q_m$.  Moreover, it follows from Lemma \ref{lem-cuculescu}(iii) and Lemma \ref{basic-lem-1}(ii) that
\[
 q_{n-1}s_n^2(v)q_{n-1}
 =q_{n-1}s_{n-1}^2(v)q_{n-1}+q_{n-1}\mathcal E_{n-1}[(dv_n)^2]q_{n-1}
 \leq 2^{2j+1}q_{n-1},
\]
and so $$p_ns_n^2(v)p_n=p_nq_{n-1}s_n^2(v)q_{n-1}p_n\leq 2^{2j+1} p_n.$$ Hence the left-hand side of
\eqref{eq:stopping-identity}, with $m=N$, is at most $2^{2j+1}\one$; namely,
\begin{equation}\label{eq:sum}
\sum_{n=2}^N q_{n-1}\mathcal E_{n-1} [(dv_n)^2]q_{n-1}\leq 2^{2j+1}\one.
\end{equation}
For a self-adjoint operator $z$ and a projection $q$, it is clear that
\[
        qz^2q-(qzq)^2=qz(\one-q)zq\geq0.
\]
Hence,
$$(dw_n)^2=(q_{n-1}dv_nq_{n-1})^2\leq q_{n-1}(dv_n)^2 q_{n-1}$$
After applying $\mathcal E_{n-1}$, this shows that
$$\mathcal E_{n-1} [(dw_n)^2] \leq q_{n-1}\mathcal E_{n-1} [(dv_n)^2]q_{n-1},$$
which combined with \eqref{eq:sum} gives the desired assertion (i) (recalling that $dw_1=0$).
\end{proof}

Now, consider the martingale $y^{(j)}=(y_n^{(j)})_{1\leq n\leq N}$ with the difference sequence given by
$$dy_1^{(j)}=0,\quad dy_n^{(j)}=dw_n^{(j)}-dw_n^{(j-1)}\,\,(2\leq n\leq N).$$
Set $e_N^{(j)}:=1-q_N^{(j)}$. Then the following assertions hold.

\begin{proposition}\label{estimate:y}
For any $j\in \mathbb Z$,
\begin{equation}\label{eq:1}
\sup_{1\leq n\leq N}\|dy_n^{(j)}\|_\infty \leq 3\cdot 2^{j},\qquad s_N^2(y)\leq 5\cdot 2^{2j}\one.
\end{equation}
Moreover,
\begin{equation}\label{eq:2}\|y_N^{(j)}\|_2^2\leq 10\cdot 2^{2j}\big(\tau (e_N^{(j)})+\tau(e_N^{(j-1)})\big)+\sum_{k=1}^N\tau\left((dx_k)^2\one_{(2^{j-1
},2^j]}(|dx_k|)\right)
\end{equation}
and
\begin{equation}\label{eq:3}\|y_N^{(j)}\|_2^2\leq 10^2\tau\Big(s_N^2(x)\one_{(2^{2j-3},\infty)}(s_N^2(x))\Big)+\sum_{k=1}^N\tau\left((dx_k)^2\one_{(2^{j-1
},2^j]}(|dx_k|)\right).
\end{equation}
\end{proposition}

\begin{proof}
Applying Lemma \ref{basic-lem-1}(i), we obtain for any $1\leq n\leq N$,
\begin{align*}\|dy_n^{(j)}\|_\infty&\leq \|dw_n^{(j)}\|_\infty+\|dw_n^{(j-1)}\|_\infty\\
&=\|q_{n-1}^{(j)}dv_n^{(j)}q_{n-1}^{(j)}\|_\infty+\|q_{n-1}^{(j-1)}dv_n^{(j-1)}q_{n-1}^{(j-1)}\|_\infty \\
&\leq \|dv_n^{(j)}\|_\infty+\|dv_n^{(j-1)}\|_\infty
\leq 2^{j+1}+2^j=3\cdot 2^j.
\end{align*}
On the other hand, for self-adjoint $a$, $b$, one has $(a-b)^2\leq 2a^2+2b^2$. Combining this  with  Lemma \ref{lem:stopping}(i), we get
$$s_N^2(y)\leq 2( s_N^2(w^{(j)})+s_N^2(w^{(j-1)}))\leq 2(2^{2j+1}\one + 2^{2j-1}\one) =5\cdot 2^j\one.$$
The assertion \eqref{eq:1} is verified.

We turn to prove \eqref{eq:2}. For $1\leq n\leq N$, put
$$r_n=q_n^{(j)}\wedge q_{n}^{(j-1)},\qquad f_n=\one-r_n.$$
The projections $r_n$ decrease in $n$ and the  $f_n$ increase. Furthermore,
\begin{equation}
\tau(f_N)=\tau(\one-r_N)\leq \tau(e_N^{(j)})+\tau(e_N^{(j-1)}).
\end{equation}
Write $$\Delta_n^j:=dw_n^{(j)}-dw_n^{(j-1)},\qquad \widetilde \Delta_n^j:=dv_n^{(j)}-dv_n^{(j-1)}.$$
Since $r_n\leq q_n^{(j)}$, $r_n\leq q_n^{(j-1)}$, we have $$r_n \Delta_n^j r_n=r_n \widetilde \Delta_n^jr_n.$$
Thus, by orthogonality, the above identity, the fact $(a-b)^2\leq 2a^2+2b^2$ for self-adjoint operators $a$, $b$, and Lemma \ref{basic:proj},
\begin{align*}
\|\Delta_n^j\|_2^2&=\|\Delta_n^j-r_n\Delta_n^jr_n\|_2^2+\| r_n\Delta_n^jr_n\|_2^2\\
& = \|(dw_n^{(j)}-r_ndw_n^{(j)}r_n)-(dw_n^{(j-1)}-r_ndw_n^{(j-1)}r_n)\|_2^2 +\| r_n \widetilde \Delta_n^j r_n\|_2^2\\
&\leq 2\|dw_n^{(j)}-r_ndw_n^{(j)}r_n\|_2^2 + 2\|dw_n^{(j-1)}-r_ndw_n^{(j-1)}r_n\|_2^2+\|\widetilde \Delta_n^j\|_2^2\\
&\leq 4\tau \big[f_{n-1}\big((dw_n^{(j)})^2 + (dw_n^{(j-1)})^2\big)\big]+\|\widetilde \Delta_n^j\|_2^2.
\end{align*}
Because $f_{n-1}\in \mathcal M_{n-1}$ and $f_{n-1}\leq f_N$, we further have
\begin{align*}
\|\Delta_n^j\|_2^2 &\leq 4 \tau \big[f_{n-1}\big(\mathcal E_{n-1}(dw_n^{(j)})^2+\mathcal E_{n-1}(dw_n^{(j-1)})^2 \big)\big]+\|\widetilde \Delta_n^j\|_2^2\\
& \leq 4 \tau \big[f_N\big(\mathcal E_{n-1}(dw_n^{(j)})^2+\mathcal E_{n-1}(dw_n^{(j-1)})^2 \big)\big]+\|\widetilde \Delta_n^j\|_2^2.
\end{align*}
Taking the sum over $n$ gives
$$\|y_N^{(j)}\|_2^2=\sum_{n=1}^N \|\Delta_n^j\|_2^2\leq 4\tau\big[f_{N}\big(s_N^2(w^{(j)})+s_N^2(w^{(j-1)})\big)\big] +\sum_{n=1}^N\|\widetilde\Delta_n^j\|_2^2.$$
According to Lemma \ref{lem:stopping}(i), $s_N^2(w^{(j)})\leq 2^{2j+1}\one$, $s_N^2(w^{(j-1)})\leq 2^{2j-1}\one$. Hence,
\begin{equation}\label{eq:leftdelta}
\|y_N^{(j)}\|_2^2\leq 10 \cdot 2^{2j}\tau(f_{N}) +\sum_{n=1}^N\|\widetilde\Delta_n^j\|_2^2.
\end{equation}
Note that $$\widetilde\Delta_n^j=dv_n^{(j)}-dv_n^{(j-1)}=dx_n\ind_{(2^{j-1},2^j]}(|dx_n|)-\mathcal E_{n-1}[dx_n\ind_{(2^{j-1},2^j]}(|dx_n|)].$$
Thus,
\begin{equation}\label{eq:widetilde}
\begin{split}
\|\widetilde\Delta_n^j\|_2^2&= \| dx_n\ind_{(2^{j-1},2^j]}(|dx_n|)\|_2^2-\|\mathcal E_{n-1}[dx_n\ind_{(2^{j-1},2^j]}(|dx_n|)]\|_2^2\\
&\leq \| dx_n\ind_{(2^{j-1},2^j]}(|dx_n|)\|_2^2.
\end{split}
\end{equation}
Plugging \eqref{eq:widetilde} into \eqref{eq:leftdelta}, we obtain \eqref{eq:2}.

Furthermore, using Lemma \ref{lem-cuculescu}(v) and Lemma \ref{basic-lem-1}(iii),
\begin{align*}
10\cdot 2^{2j}\big(\tau (e_N^{(j)})+\tau(e_N^{(j-1)})\big)&\leq 2\tau\big(s_N^2(v^{(j)})\one_{(2^{2j-1},\infty)}(s_N^2(v^{(j)}))\big)\\
&\qquad + 8\tau\big(s_N^2(v^{(j-1)})\one_{(2^{2j-3},\infty)}(s_N^2(v^{(j-1)}))\big)\\
&\leq 10 \tau\big(s_N^2(v^{(j)})\one_{(2^{2j-3},\infty)}(s_N^2(v^{(j)}))\big)\\
&\leq  10 \tau\big(s_N^2(x)\one_{(2^{2j-3},\infty)}(s_N^2(x))\big).
\end{align*}
This, together with \eqref{eq:2}, proves \eqref{eq:3}.
\end{proof}

Combining Proposition \ref{estimate:y} with Proposition \ref{lem:intrinsic}, the following localization estimate is immediate.
\begin{cor}\label{cor:localized}
Fix $j\in\mathbb Z$. For  $4<p<\infty$, we have
\begin{equation}\label{eq:localized}
\|y_N^{(j)}\|_p^p \leq \left(\frac{32p}{\log p}\right)^p 2^{j(p-2)}\|y_N^{(j)}\|_2^2.
\end{equation}
\end{cor}

\section{Dyadic summation}
This section assembles the summability estimates needed to pass from dyadic localization to the global inequality. We control the weighted sums of the localized \(L^2\)-energies, the traces of the exceptional projections, and the truncation errors by the \(L^p\)-norm of the conditioned square function and the sum of the \(p\)-th moments of the martingale differences. We then reconstruct each stopped approximation from the localized pieces in \(L^{2p}\). This choice of exponent produces a summable convolution kernel independent of \(p\), preserving the \(p/\log p\) dependence obtained in Section 3.

\begin{lemma}\label{dyadic-1} For $4\leq p<\infty$, we have
$$\sum_{j\in\mathbb Z}2^{j(p-2)}\|y_N^{(j)}\|_2^2\leq C_{\rm abs}^p\Big(\|s_N(x)\|_p^p+\sum_{k=1}^N\|dx_k\|_p^p\Big).$$
\end{lemma}

\begin{proof}
For $\alpha$, $L>0$,  we have
\begin{equation}\label{eq:series}
\sum_{\{j\in\mathbb Z: 2^j<L\} } 2^{j\alpha}\leq \frac{L^\alpha}{1-2^{-\alpha}}.
\end{equation}
Let $d\mu(t):=d\tau(\one_{(-\infty,t]}(s_N^2(x)))$ denote the scalar spectral measure induced by $s_N^2(x)$ and $\tau$. For $c>0$, by Tonelli's theorem and \eqref{eq:series},
\begin{equation}\label{eq:ess-1}
\begin{split}
&\sum_{j\in\mathbb Z} 2^{j(p-2)} \tau(s_N^2(x)\one_{(2^{2j}/c,\infty)}(s_N^2(x)))\\&\qquad=\int_0^\infty t\Big(\sum_{\{j\in\mathbb Z:2^j<\sqrt{ct}\}}2^{j(p-2)}\Big)d\mu(t)\\
&\qquad\leq \frac{c^{(p-2)/2}}{1-2^{2-p}}\int_0^\infty t^{p/2} d\mu(t)=\frac{c^{(p-2)/2}}{1-2^{2-p}}\|s_N(x)\|_p^p.
\end{split}
\end{equation}
On the other hand,
\begin{equation}\label{eq:ess-2}
\begin{split}
&\sum_{j\in\mathbb Z} 2^{j(p-2)} \sum_{k=1}^N\tau\left((dx_k)^2\one_{(2^{j-1
},2^j]}(|dx_k|)\right)\\
&\qquad \leq 2^{p-2}\sum_{k=1}^N \tau(|dx_k|^p)=2^{p-2}\sum_{k=1}^N\|dx_k\|_p^p.
\end{split}
\end{equation}
Applying \eqref{eq:3}, \eqref{eq:ess-1} with $c=8$, and \eqref{eq:ess-2}, we deduce the desired estimate.
\end{proof}

\begin{lemma}\label{dyadic-2} For $4\leq p<\infty$, we have
$$\sum_{j\in\mathbb Z}2^{pj}\tau(e_N^{(j)})\leq C_{\rm abs}^p \|s_N(x)\|_p^p.$$
\end{lemma}

\begin{proof}
Applying Lemma \ref{lem-cuculescu}(v), Lemma \ref{basic-lem-1}(iii) and \eqref{eq:ess-1} with $c=2$, we conclude
\begin{align*}
\sum_{j\in\mathbb Z}2^{jp}\tau(e_N^{(j)})&\leq 2\sum_{j\in\mathbb Z} 2^{(p-2)j}\tau\big(s_N^2(v^{(j)})\one_{(2^{2j-1},\infty)}(s_N^2(v^{(j)}))\big)\\
& \leq 2\sum_{j\in\mathbb Z} 2^{(p-2)j}\tau\big(s_N^2(x)\one_{(2^{2j-1},\infty)}(s_N^2(x))\big)\\
&\leq \frac{2^{p/2}}{1-2^{2-p}}\|s_N(x)\|_p^p.
\end{align*}
The proof is complete.
\end{proof}

\begin{lemma} \label{dyadic-3} For $4\leq p<\infty$, we have
$$\sum_{j\in\mathbb Z} 2^{j(p-1)}\sum_{k=1}^N\tau(|dx_k|\one_{(2^{j},\infty)}(|dx_k|))\leq C_{\rm abs} \sum_{k=1}^N\|dx_k\|_p^p.$$
\end{lemma}

\begin{proof}
Let $d\mu_k$ denote the scalar  spectral measure of $|dx_k|$. Applying Tonelli's theorem and \eqref{eq:series}, 
\begin{align*}
\sum_{j\in\mathbb Z} 2^{j(p-1)}\sum_{k=1}^N\tau(|dx_k|\one_{(2^{j},\infty)}(|dx_k|))&=\sum_{k=1}^N \int_0^\infty t \sum_{\{j\in\mathbb Z:2^j<t\}} 2^{(p-1)j}\, d\mu_k(t)\\
&\leq \frac{1}{1-2^{1-p}}\sum_{k=1}^N \int_0^\infty t^{p}\, d\mu_k(t)\\
&=\frac{1}{1-2^{1-p}} \sum_{k=1}^N\|dx_k\|_p^p,
\end{align*}
which proves the desired assertion.
\end{proof}

We next recombine all dyadic pieces below a fixed level. The use of the exponent
$2p$ is what produces a summable convolution kernel independent of $p$.

\begin{lemma}\label{lem:w}
Let $4\leq p<\infty$. We have
$$\sum_{j\in\mathbb Z} 2^{-pj}\|w_N^{(j)}\|_{2p}^{2p}
\leq C_{abs}^{2p}\left(\frac{2p}{\log (2p)}\right)^{2p}\sum_{j\in\mathbb Z}2^{j(p-2)}\|y_N^{(j)}\|_2^2.$$
\end{lemma}

\begin{proof}
Introduce two sequences $a=(a_i)_{i\in\mathbb Z}$ and $b=(b_k)_{k\in\mathbb N}$:
$$a_i:=2^{-i/2}\|y_N^{(i)}\|_{2p},\qquad b_k:=2^{-k/2}.$$
Clearly, $b=(b_k)_{k\in\mathbb N}\in \ell_{(2p)'}$, where $(2p)'$ denotes the conjugate number of $2p$. Also, by Corollary \ref{cor:localized} and Lemma \ref{dyadic-1}, one can easily see  $a=(a_i)_{i\in\mathbb Z}\in \ell_{2p}$. Consequently, for fixed $j\in\mathbb Z$, 
\begin{equation}\label{eq:series-y}
\sum_{i\leq j}\|y_N^{(i)}\|_{2p}=2^{j/2}\sum_{k=0}^\infty b_ka_{j-k}\leq 2^{j/2}\|(b_k)_{k\in\mathbb N}\|_{(2p)'} \|(a_i)_{i\in\mathbb Z}\|_{2p}<\infty,
\end{equation}
where H\"{o}lder's inequality is used. This means that for each $j\in\mathbb Z$, the series $\sum_{i\leq j} y_N^{(i)}$ converges in $L^{2p}.$ Note that the $m$-th partial sum can be written as follows:
$$\sum_{i=m}^jy_N^{(i)} =w_N^{(j)}-w_N^{(m-1)}.$$
By Lemma \ref{basic-lem-1}(i), $\|w_N^{(m)}\|\leq \sum_{n=1}^N\|dw_n^{(m)}\|_\infty\leq \sum_{n=1}^N\|dv_n^{(m)}\|_\infty\leq (N-1)2^{m+1}$. Since the trace is finite, $w_N^{(m)}\to 0$ in $L^{2p}$ as $m\to\infty$. Hence, the series $\sum_{i\leq j} y_N^{(i)}$ equals $w_N^{(j)}$, and \eqref{eq:series-y} implies
$$2^{-j/2}\|w_N^{(j)}\|_{2p}\leq \sum_{k=0}^\infty b_ka_{j-k}=b\ast a.$$
By Minkowski's inequality, we obtain
\begin{equation}\label{eq:almost}\sum_{j\in\mathbb Z}2^{-pj}\|w_N^{(j)}\|_{2p}^{2p}\leq \|b\ast a\|_{2p}^{2p}\leq \|b\|_1^{2p}\|a\|_{2p}^{2p}\leq C_{\rm abs}\sum_{j\in\mathbb Z}2^{-pj}\|y_N^{(j)}\|_{2p}^{2p}.\end{equation}
According to Corollary \ref{cor:localized},
$$\|y_N^{(j)}\|_{2p}^{2p} \leq C_{\rm abs}^{2p}\left(\frac{2p}{\log 2p}\right)^{2p} 2^{j(2p-2)}\|y_N^{(j)}\|_2^2.$$
Plugging this estimate into \eqref{eq:almost}, we conclude the desired result.
\end{proof}

\section{Proof of the main result}

In this section, we combine all materials established in the previous sections to provide the proof of Theorem \ref{main:result}.  Before that,  we establish the following lemma which compares the original terminal value with its stopped
approximation.

\begin{lemma}\label{dis-est}
For every $j\in \mathbb Z$ and $s>0$, the following inequality 
$$\lambda_{2^{j+1}s}(x_N)\leq \lambda_{2^js}(w_N^{(j)})+2\tau(e_N^{(j)})+2^{1-j}s^{-1}\sum_{k=1}^N\tau(|dx_k|\one_{(2^{j},\infty)}(|dx_k|))$$
holds true.
\end{lemma}
\begin{proof}
By the property of martingale, we have 
\begin{equation}\label{6.11}
\begin{split}
\|x_N-v_N^{(j)}\|_1&=\left\|\sum_{k=2}^Ndx_k \one_{(2^{j},\infty)}(|dx_k|)- \sum_{k=2}^N \mathcal E_{k-1}\left(dx_k \one_{(2^{j},\infty)}(|dx_k|)\right)\right\|_1 \\
&\leq 2\sum_{k=1}^N\tau(|dx_k|\one_{(2^{j},\infty)}(|dx_k|)).
\end{split}
\end{equation}
Using Lemma \ref{lem:stopping}(ii), we have $q_N^{(j)}(v_N^{(j)}-w_N^{(j)})q_N^{(j)}=0$. Thus, it follows from Lemma \ref{lem:tails} that
$$\tau({\rm supp}(v_N^{(j)}-w_N^{(j)}))\leq 2\tau (e_N^{(j)}),$$
and for arbitrary $t>0$, 
$$\lambda_t (v_N^j)\leq \lambda_t (w_N^j)+\tau({\rm supp}(v_N^{(j)}-w_N^{(j)}))\leq
\lambda_t (w_N^j)+2\tau (e_N^{(j)}).$$
Applying \eqref{distribution-triangle}, it follows from the above estimate  that
\begin{align*}
\lambda_{2^{j+1}s}(x_N)&\leq \lambda_{2^{j}s}(v_N^{(j)})+\lambda_{2^{j}s}(x_N-v_N^{(j)})\\
&\leq \lambda_{2^js}(w_N^{(j)})+2\tau(e_N^{(j)})+\lambda_{2^{j}s}(x_N-v_N^{(j)}).
\end{align*}
Applying  Chebyshev's inequaltiy and  \eqref{6.11}, we get
\begin{align*}
\lambda_{2^{j+1}s}(x_N)&\leq \lambda_{2^js}(w_N^{(j)})+2\tau(e_N^{(j)})+
2^{-j} s^{-1}\|x_N-v_N^{(j)}\|_1\\
&\leq \lambda_{2^js}(w_N^{(j)})+2\tau(e_N^{(j)})+2^{1-j}s^{-1}\sum_{k=1}^N\tau(|dx_k|\one_{(2^{j},\infty)}(|dx_k|)),
\end{align*}
which completes the proof.
\end{proof}

We are now ready to show the main result.

\begin{proof}[Proof of Theorem \ref{main:result}]
It suffices to prove the result for finite self-adjoint martingales $(x_n)_{1\leq n\leq N}$ with the initial value $x_1=0$.  Indeed, if $x_1\neq 0$, then one can instead consider the martingale $(x_n-x_1)_{1\leq n\leq N}$. Set 
$s=\frac{2p}{\log (2p)}L$, where $L>0$ is  a sufficiently large absolute constant.
Since the distribution function of $x_N$ is nonincreasing, it follows that 
\begin{equation*}\label{moment-x}
\begin{split}
\|x_N\|_p^p &=\sum_{j\in\mathbb Z}p\int_{2^{j+1}s}^{2^{j+2}s} t^{p-1} \lambda_t(x_N) \ dt \\
&\leq (2^p-1)(2s)^p\sum_{j\in\mathbb Z}  2^{pj}\lambda_{2^{j+1}s}(x_N).
\end{split}
\end{equation*}
Applying Lemma \ref{dis-est}, we deduce that
\begin{equation}\label{eq:123}
\|x_N\|_p^p \leq (2^p-1)(2s)^p \Big( {\rm I}+{\rm II}+{\rm III}\big)\Big),
\end{equation}
where
\begin{align*}
{\rm I}&:=\sum_{j\in \mathbb Z} 2^{pj}\lambda_{2^js}(w_N^{(j)}),\\
{\rm II}&:=2\sum_{j\in \mathbb Z} 2^{pj}\tau(e_N^{(j)}),\\
{\rm III}&:=2s^{-1} \sum_{j\in \mathbb Z} 2^{(p-1)j}\sum_{k=1}^N\tau(|dx_k|\one_{(2^{j},\infty)}(|dx_k|)).
\end{align*}
For ${\rm I}$, it follows from the Chebyshev inequality, Lemma \ref{lem:w} and Lemma \ref{dyadic-1}  that
\begin{align*}
{\rm I}&\leq s^{-2p}\sum_{j\in\mathbb Z}2^{-pj}\|w_N^{(j)}\|_{2p}^{2p}\\
&\leq s^{-2p}  C_{abs}^{2p}\left(\frac{2p}{\log (2p)}\right)^{2p}\sum_{j\in\mathbb Z}2^{j(p-2)}\|y_N^{(j)}\|_2^2\\
& \leq C_{abs}^{2p}  \Big(\|s_N(x)\|_p^p+\sum_{k=1}^N\|dx_k\|_p^p\Big).
\end{align*}
By Lemma \ref{dyadic-2}, 
$${\rm II}\leq  C_{\rm abs}^p \|s_N(x)\|_p^p.$$
Applying Lemma \ref{dyadic-3}, we derive
$${\rm III}\leq  C_{\rm abs} s^{-1} \sum_{k=1}^N\|dx_k\|_p^p\leq C_{\rm abs} \sum_{k=1}^N\|dx_k\|_p^p,$$
where the last inequality holds since $L$ is sufficiently large.
Combining the above estimates for ${\rm I}$,  ${\rm II}$,  ${\rm II}$ with \eqref{eq:123}, we conclude
$$\|x_N\|_p^p\leq (2^p-1)(2s)^p C_{abs}^{p}  \Big(\|s_N(x)\|_p^p+\sum_{k=1}^N\|dx_k\|_p^p\Big).$$
Taking $p$-th root  gives
$$\|x_N\|_p\leq C_{\rm abs}\frac{p}{\log p}\max\Big\{\|s_N(x)\|_p, \Big(\sum_{k=1}^N\|dx_k\|_p^p\Big)^{1/p}\Big\}.$$
The proof is complete.
\end{proof}


\section*{Acknowledgment}
During the preparation of this work, the authors used GPT-5.6 Sol to assist with language editing and refinement of mathematical arguments. The authors take full responsibility for the content of this paper.

\end{document}